\documentclass[11pt]{amsproc}
\usepackage{srcltx}
\usepackage{amssymb}

\newcommand\R{\mathbb{R}}
\newcommand\Z{\mathbb{Z}}
\newcommand\N{\mathbb{N}}
\newcommand\T{\mathbb{T}}
\newcommand\cW{\mathcal{W}}
\newcommand\wh{\widehat}

\newcommand\ol{\overline}

\newcommand\supp{\operatorname{supp}}

\newcommand\diag{\operatorname{diag}}

\theoremstyle{plain}
\newtheorem*{thma}{Theorem~4$'$}
\newtheorem{thm}{Theorem}
\newtheorem{proposition}[thm]{Proposition}
\newtheorem{lem}{Lemma}
\newtheorem{cor}[thm]{Corollary}
\theoremstyle{remark}
\newtheorem{rem}{Remark}

\begin{document}

\title{Wiener's problem for positive definite functions}

\author{D.~V.~Gorbachev}
\address{D.~Gorbachev, Tula State University,
Department of Applied Mathematics and Computer Science,
300012 Tula, Russia}
\email{dvgmail@mail.ru}

\author{S.~Yu.~Tikhonov}
\address{S.~Tikhonov, ICREA, Centre de Recerca Matem\`{a}tica,
Campus de Bellaterra, Edifici~C 08193 Bellaterra (Barcelona), Spain, and UAB}
\email{stikhonov@crm.cat}


\thanks{D.~G. was supported by the RFBR (no.~16-01-00308), the Ministry of
Education and Science of the Russian Federation (no.~5414GZ), and by D.~Zimin's
Dynasty Foundation. S.~T. was partially supported by MTM 2014-59174-P, 2014 SGR
289, and by the Alexander von Humboldt Foundation.}

\begin{abstract}
We study the sharp constant $W_{n}(D)$ in Wiener's inequality for positive
definite functions
\[
\int_{\mathbb{T}^{n}}|f|^{2}\,dx\le W_{n}(D)|D|^{-1}\int_{D}|f|^{2}\,dx,\quad
D\subset \mathbb{T}^{n}.
\]
N. Wiener proved that $W_{1}([-\delta,\delta])<\infty$, $\delta\in (0,1/2)$.
E.~Hlawka showed that $W_{n}(D)\le 2^{n}$, where $D$ is an origin-symmetric
convex body.

We sharpen Hlawka's estimates for $D$ being the ball $B^{n}$ and the cube $I^{n}$.
In particular, we prove that $W_{n}(B^{n})\le 2^{(0.401\ldots +o(1))n}$. We also
obtain a lower bound of $W_{n}(D)$. Moreover, for a cube $ D=\frac1q I^{n}$ with
$q=3,4,\ldots,$ we obtain that $W_{n}(D)=2^{n}$. Our proofs are based on the
interrelation between Wiener's problem and the problems of Tur\'an and
Delsarte.
\end{abstract}
\maketitle

\section{Introduction}

Let $n\in \N$ and $\T^{n}=\R^{n}/\Z^{n}$. The Fourier series of a
complex-valued function $f\in L^{1}(\T^{n})$ is given by
\[
f(x)=\sum_{\nu\in \Z^{n}}\wh{f}_{\nu}e(\nu x),\qquad e(t)=e^{2\pi it},
\]
where
\[
\wh{f}_{\nu}=\int_{\T^{n}}f(x)e(-\nu x)\,dx,\quad \nu\in \Z^{n},
\]
are the Fourier coefficients of $f$. The support of a function $f$, written
$\supp f$, is the closure of the subset of $\T^{n}$ where $f$ is non-zero. Let
the unit ball and the unit cube be given by $B^{n}=\{x\in \R^{n}\colon |x|\le
1\}$ and $I^{n}=[-1,1]^{n}$, respectively. Let also $B_{r}:=B_{r}^{n}:=rB^{n}$
for $r>0$. By $|D|$ we denote the volume of $D\subset \R^{n}$. In what follows,
we assume that $D$ is an origin-symmetric convex body.

Wiener's inequality for positive definite functions in $L^{2}(\T^{n})$ is given
by
\begin{equation}\label{wiener-f}
\int_{\T^{n}}|f(x)|^{2}\,dx\le C_{n}(D)\int_{D}|f(x)|^{2}\,dx,\quad f\in
L_{+}^{1}(\T^{n}),
\end{equation}
where
\[
L_{+}^{1}(\T^{n}):=\bigl\{f\in L^{1}(\T^{n})\colon \wh{f}_{\nu}\ge 0\ \text{for
any}\ \nu\in \Z^{n}\bigr\},
\]
and $D\subset \T^{n}$. Here $C_{n}(D)$ is a positive constant depending only on
$n$ and~$D$. Note that $L_{+}^{1}(\T^{n})\nsubseteq L^{2}(\T^{n}),$ take, for
example,
\[
f(x)=\sum_{k=1}^\infty k^{-1/2}\cos{}(2\pi kx_{1}),
\]
see \cite[Ch.~V, (1.8)]{zy02}.

N.~Wiener (unpublished result, see e.g.~\cite{sh75}) proved in the early 1950's
that $C_{1}([-\delta,\delta])<\infty$ for $\delta\in (0,1/2)$.

For $n=1$, H.~Shapiro~\cite{sh75} showed that, for any $\delta\in (0,1/2)$,
\[
\int_{\T}|f|^{2}\,dx\le \delta^{-1}\int_{-\delta}^{\delta}|f|^{2}\,dx,\quad
f\in L_{+}^{1}(\T).
\]
The latter was generalized by E.~Hlawka \cite{hl81} for the multivariate case
as follows:
\begin{equation}\label{hlawka-ineq}
\int_{\T^{n}}|f|^{2}\,dx\le
\bigl|\tfrac{1}2D\bigr|^{-1}\int_{D}|f|^{2}\,dx,\quad f\in L_{+}^{1}(\T^{n}),
\end{equation}
where $D\subset \T^{n}$ is an origin-symmetric convex body.

The goal of this paper is to study the sharp constant
\[
W_{n}(D):=\sup_{f\in L_{+}^{1}(\T^{n})\setminus
\{0\}}\frac{\int_{\T^{n}}|f|^{2}\,dx}{|D|^{-1}\int_{D}|f|^{2}\,dx}.
\]
Note that Hlawka's result implies that
\begin{equation}\label{zv-2page}
W_{n}(D)\le 2^{n},\quad n\in \N.
\end{equation}
Moreover, taking $f=1$ we get a trivial estimate from below
\[
1\le W_{n}(D)
\]
for any $D\subset \T^{n}$.

Note that $f\in L_{+}^{1}(\T^{n})$ if and only if $f$ is positive definite
\cite[9.2.4]{ed79}. Recall that an integrable function $f$ is positive definite
\cite[Chap.~9]{ed79} if
\begin{equation}\label{u-cond}
\int_{\T^{n}}\int_{\T^{n}}f(x-y)u(x)\ol{u(y)}\,dx\,dy\ge 0
\end{equation}
for any $u\in C(\T^{n})$. It is sufficient to verify (\ref{u-cond}) only for
the case of $u$ being trigonometric polynomials.

For a continuous function $f\in C(\T^{n})$, condition \eqref{u-cond} is
equivalent to the fact that $f$ is positive definite in the classical sense,
that is, for every finite sequence $\{x_{i}\}_{i=1}^{N}$ in $\T^{n}$ and every
choice of complex numbers $\{c_{i}\}_{i=1}^{N}$, we have
\[
\sum_{i,j=1}^{N}c_{i}\ol{c_{j}}f(x_{i}-x_{j})\ge 0,
\]
see \cite[Chap.~1]{ru62}.

Note that if a function $f\in L_{+}^{1}(\T^{n})$ is bounded in some
neighborhood of the origin and therefore the series $\sum_{\nu}\wh{f}_{\nu}$
converges, then Bochner's theorem \cite[9.2.8]{ed79} implies that $f$ can be
viewed as a continuous positive definite function. In general, this is not the
case. However, the following result is true (see Section~\ref{sec-proofs} for
its proof).

\begin{proposition}\label{main0}
We have
\begin{equation}\label{W-F+}
W_{n}(D)=W_{n}^{+}(D):=\sup_{f\in
F_{+}\setminus
\{0\}}
\frac{\int_{\T^{n}}|f|^{2}\,dx}{|D|^{-1}\int_{D}|f|^{2}\,dx},
\end{equation}
where $F_{+}$ is a set of continuous positive definite functions on $\T^{n}$.
\end{proposition}

In this paper, we continue investigating multivariate inequality
\eqref{hlawka-ineq} and prove new bounds for $W_{n}(D)$. Moreover, we connect
this problem to Tur\'an's problem (see Section~\ref{sec-tur}) and Delsarte's
problem also known as the linear programming bound problem (see
Section~\ref{sec-del}).

The main results of the paper are new bounds of $ W_{n}(D)$ in the case when
$D$ is a ball or a cube.

\begin{thm}\label{main1}
For $\delta\in (0,1/2)$, we have
\[
W_{n}(\delta B^{n})\le 2^{(0.401\ldots +o(1))n}.
\]
\end{thm}

\begin{thm}\label{main2}
Let $\delta\in (0,1/2)$. Then
\[
W_{n}(\delta I^{n})\le 2^{n}(1-\theta(\delta))^{n},
\]
where
\begin{itemize}
\item[(\textit{i})]\quad $\theta(\delta)\in C(0, 1/2)$, and, moreover,
$\theta(\delta)=O(\delta^{2})$ as $\delta\to 0$,
\item[(\textit{ii})]\quad $\theta(\delta)=0$ for $\delta^{-1}\in \N$,
\item[(\textit{iii})]\quad $0<\theta(\delta)<1$ for $\delta^{-1}\notin \N$.
\end{itemize}
\end{thm}
Note that there is a specific expression for $\theta(\delta)$, which is
\[
\theta(\delta)=1-\delta a_{\T}^{-1}([-\delta,\delta]),
\]
where $a_{\T}([-\delta,\delta])$ is the solution of Tur\'an's problem, see
Section~\ref{sec-tur}.

Our next result provides an estimate of $ W_{n}(D)$ from below.

\begin{thm}\label{main3}
Let $D\subset \delta I^{n}$ with $\delta\in (0,1/q)$ for some $q=2,3,\dots$. Then
\[
|D|q^{n}\le W_{n}(D).
\]
\end{thm}
For the cube $D=\delta I^{n}$, this result gives
\begin{equation}\label{W-F--}
2^{n}(\delta q)^{n}\le W_{n}(\delta I^{n}),\quad \delta\in (0,1/q).
\end{equation}

Letting $\delta\to \frac1q-$, this and Hlawka's inequality \eqref{zv-2page}
give
\begin{cor}[Wiener's constant for the cube]
For $q=3,4,\ldots,$ we have
\[
W_{n}(\tfrac{1}{q}I^{n})=2^{n}.
\]
\end{cor}
It is worth mentioning that setting\footnote{As usual, by $[a]$ we denote the
integer part of $a\in \R$.} $q=[\delta^{-1}]$ for $\delta^{-1}\notin \N$ and
$q=[\delta^{-1}]-1$ for $\delta^{-1}\in \N$ in (\ref{W-F--}), we obtain
\[
2^{n}(1-\delta)^{n}\le W_{n}(\delta I^{n}),\quad \delta\in (0,1/2).
\]
In particular, this implies that the limit
\[
\lim_{\delta\to 0+}W_{n}(\delta I^{n})=2^{n}
\]
exists. It is not known if this limit exists for other $D$.

\section{Wiener's problem vs. Tur\'an's problem}\label{sec-tur}

The periodic Tur\'an problem in $L^{1}(\T^{n})$ for positive definite functions
consists of finding \cite{go01}
\[
a_{\T^{n}}(D)=\sup \wh{g}_{0},
\]
where supremum is taken over all functions $g\in L^{1}(\T^{n})$ such that
\begin{equation}\label{turan}
\wh{g}_{\nu}\ge 0,\quad \supp g\subset D,\quad g(0)=1.
\end{equation}

Similarly we introduce the non-periodic Tur\'an problem in $L^{1}(\R^{n})$ for
positive definite functions with compact support:
\[
a_{\R^{n}}(D)=\sup\wh{u}(0),
\]
where
\[
h\in L^{1}(\R^{n}),\quad \wh{h}\ge 0,\quad \supp h\subset D,\quad h(0)=1.
\]
Here $D$ is any subset of $\R^{n}$ and
\[
\wh{h}(\xi)=\int_{\R^{n}}h(x)e(-\xi x)\,dx,\quad \xi\in \R^{n},
\]
is the Fourier transform of $h$.

Note that Tur\'an's problem is closely related to Boas--Kac--Krein
representation theorem on convolution roots of positive definite functions with
compact support, which in turn has applications to geostatistical simulation,
crystallography, optics, and phase retrieval (see \cite{ehgnri04}) as well as
Fuglede's conjecture (see \cite{kore03}). Very recent results on the topic can
be found in \cite{kore06a}.

The periodic and spatial problems are connected as follows \cite{go01}:
\[
a_{\R^{n}}(D)\le \lambda^{-n}a_{\T^{n}}(\lambda D)\le
a_{\R^{n}}(D)(1+O(\lambda^{2})),\quad \lambda\in (0,1].
\]

The periodic Tur\'an problem was studied in one dimension in
\cite{goma04,ivgoru05,ivru05} and completely solved in \cite{iv06}. Since the
solution has a complicated form, we only highlight the following two facts. If
$q=2,3,\dots$, then
\begin{equation}\label{ste-eq}
a_{\T}([-1/q,1/q])=1/q
\end{equation}
(see also \cite{ste}).

If $\delta\in (0,1/2)$, then
\begin{equation}\label{a-prop}
\begin{aligned}
&a_{\T}([-\delta,\delta])>\delta,&&\delta^{-1}\notin \N,
\\
&a_{\T}([-\delta,\delta])=\delta(1+O(\delta^{2})),&&\delta\to 0.
\end{aligned}
\end{equation}

Now we are in a position to sharpen the known bounds of $W_{n}(D)$ from above,
cf.~\eqref{zv-2page}.
\begin{thm}\label{th-turan}
Let $D\subset \T^{n}$. We have
\[
|D|^{-1}W_{n}(D)\le (a_{\T^{n}}(D))^{-1}\le (a_{\R^{n}}(D))^{-1}
\le |\tfrac{1}{2}D|^{-1}.
\]
\end{thm}

\begin{proof}
First, we show that
\begin{equation}\label{W-a-ineq}
|D|^{-1}W_{n}(D)\le (a_{\T^{n}}(D))^{-1}.
\end{equation}

Let $g$ be an admissible function for the periodic Tur\'an problem, i.e., $g$
satisfies condition \eqref{turan}. Then since $g$ is positive definite, we have
$\supp g\subset D$ and $g(x)\le g(0)=1$, $x\in \T^{n}$. Hence for any $f\in
L_{+}^{1}(\T^{n})$ we get
\[
\int_{D}|f|^{2}\,dx\ge \int_{\T^{n}}|f|^{2}g\,dx.
\]
Note that $fg\in L^{1}(\T^{n})\cap L^{2}(\T^{n})$. Since both $f$ and $g$ have
nonnegative Fourier coefficients, we obtain that
\[
(\wh{fg})_{\nu}=\sum_{\mu}\wh{f}_{\nu-\mu}\wh{g}_{\mu}\ge
\wh{g}_{0}\wh{f}_{\nu}
\]
and
\begin{align}\notag
\int_{\T^{n}}|f|^{2}g\,dx
&=\int_{\T^{n}}\ol{f(x)}\sum_{\nu}(\wh{fg})_{\nu}e(\nu x)\,dx=
\sum_{\nu}(\wh{fg})_{\nu}\int_{\T^{n}}\ol{f(x)}e(\nu x)\,dx\\
\label{2zv}
&=\sum_{\nu}(\wh{fg})_{\nu}\ol{\wh{f}_{\nu}}\ge \wh{g}_{0}\sum_{\nu}|\wh{f}_{\nu}|^{2}=
\wh{g}_{0}\int_{\T^{n}}|f|^{2}\,dx.
\end{align}
Thus,
\begin{equation}\label{1zv}
\int_{\T^{n}}|f|^{2}\,dx\le (\wh{g}_{0})^{-1}\int_{D}|f|^{2}\,dx,
\end{equation}
which gives $|D|^{-1}W_{n}(D)\le (\wh{g}_{0})^{-1}$. Minimizing
$(\wh{g}_{0})^{-1}$, or equivalently maximizing~$\wh{g}_{0}$, we arrive at
\eqref{W-a-ineq}.

Secondly, we prove that $a_{\R^{n}}(D)\le a_{\T^{n}}(D)$. Let $h$ be any
admissible function in the spatial Tur\'an problem. Consider the periodic
function \cite[Chap.~7]{stwe71}
\[
g(x)=\sum_{\nu\in \Z^{n}}h(x+\nu).
\]
Then $\wh{g}_{\nu}=\wh{h}(\nu)\ge 0$, $\supp g\subset D$, and
$g(0)=\sum_{\nu\in \Z^{n}}h(\nu)=h(0)=1$. Therefore, $g$ is an admissible
function in the periodic Tur\'an problem and, moreover,
$\wh{h}(0)=\wh{g}_{0}\le a_{\T^{n}}(D)$, which implies $a_{\R^{n}}(D)\le
a_{\T^{n}}(D)$.

Third, to show that $(a_{\R^{n}}(D))^{-1}\le |K|^{-1},$ where
$K=\tfrac{1}{2}D$, we set
\begin{equation}\label{h*}
h_{*}=b^{-1}\chi_{K}*\chi_{K},
\end{equation}
where $b=(\chi_{K}*\chi_{K})(0)=|K|$. Then $\supp h_{*}\subset D$ and
$\wh{h}_{*}=b^{-1}\wh{\chi}_{K}^{2}\ge 0$. Moreover,
\begin{equation}\label{aR-chi}
a_{\R^{n}}(D)\ge \wh{h}_{*}(0)=|K|^{-1}\wh{\chi}_{K}^{2}(0)
=|K|^{-1}\Bigl(\int_K dx\Bigr)^{2}=\bigl|\tfrac{1}2D\bigr|,
\end{equation}
which completes the proof.
\end{proof}

There is a conjecture \cite{bike14} that
\[
a_{\R^{n}}(D)=\bigl|\tfrac{1}2D\bigr|
\]
and $h_{*}$ is an extremal function. This conjecture was proved only for the
ball and Voronoi polytopes of lattices. In the case of the ball this was first
done by Siegel \cite{si35} and later in \cite{go01,kore03,bike14}). For the
case of the Voronoi polytopes see \cite{arbe01,arbe02}. It is worth mentioning
that these results are also known for any rotation and scaling of $D$. This
follows from
\begin{rem}
Let $f$ be an admissible function in Tur\'an problem, $\rho\in SO(n)$,
$\lambda=\diag{}(\lambda_{1},\dots,\lambda_n)$, $\lambda_i> 0$, and
$g(x)=f(\lambda^{-1}\rho^{-1}x),$ $x\in \R^{n}$. Then $\supp g\subset
\rho\lambda D$ and $\widehat{g}(\xi)=\det \lambda\cdot
\widehat{f}(\lambda\rho^{-1}\xi)\ge 0,$ where $\xi\in \R^{n}$ and $\det
\lambda=\lambda_{1}\dots\lambda_n$.

We have that $g$ is a positive definitive function such that
$\widehat{g}(0)=\det \lambda\cdot \widehat{f}(0)$ and $g(0)=f(0)=1$. Thus, we
get
\[
a_{\R^{n}}(\lambda\rho D)= \det \lambda\cdot a_{\R^{n}}(D).
\]
\end{rem}

Theorem \ref{th-turan} does not provide an improvement of Hlawka's inequality
$W_{n}(D)\le 2^{n}$ in the case of the ball. In order to sharpen this bound we
will now consider a wider class of admissible functions $h$.

\section{Wiener's problem vs. Delsarte's problem}\label{sec-del}

By the Delsarte problem in $L^{1}(\R^{n})$ for positive definite functions we
mean the following question \cite{go00,coel03}:
\begin{equation}\label{dels}
A_{\R^{n}}(D)=\inf h(0),
\end{equation}
where
infimum is taken over all functions $h\in L^{1}(\R^{n})$ such that
\[
\wh{h}\ge 0,\quad h|_{\R^{n}\setminus D}\le 0,\quad \wh{h}(0)=1.
\]
Noting that $\wh{h}_{\lambda}(\xi)=\lambda^{-n}\wh{h}(\lambda^{-1}\xi)$, where
$h_{\lambda}(x)=h(\lambda x)$, we have
\begin{equation}\label{A-hom}
A_{\R^{n}}(\lambda D)=\lambda^{-n}A_{\R^{n}}(D),\quad \lambda>0.
\end{equation}

The importance of Delsarte's problem can be illustrated by the following
remarks. Let $D=2B^{n}$ and $\Delta_{n}$ be the center density of sphere
packings in~$\R^{n}$ (see the details in \cite{go00,coel03,co02}).

The best known density bounds are
\[
2^{-(1+o(1))n}\le \Delta_{n}\le 2^{-(0.5990\ldots +o(1))n}=:C_{KL}
\]
as $n\to \infty$. For the best known lower estimate see \cite{ve13}. The upper
bound was found by G.~Kabatiansky and V.~Levenshtein in \cite{kale78}.

In \cite{le79}, V.~Levenshtein proved that
\[
\Delta_{n}\le \frac{(q_{n/2}/4)^{n}}{\Gamma^{2}(n/2+1)}=2^{-(0.5573\ldots
+o(1))n}=:C_{L},
\]
where $q_{n/2}$ is the first positive zero of the Bessel function $J_{n/2} (t)$.

Later, D.~Gorbachev \cite{go00} and H.~Cohn and N.~Elkies \cite{co02,coel03}
proved the linear programming bounds
\[
\Delta_{n}\le |B^{n}|A_{\R^{n}}(2B^{n})=:C_{A}.
\]
In particular, this yields that
\[
C_{A}\le C_{L}.
\]
Moreover, for admissible functions $h$ for the Delsarte problem, that is
satisfying \eqref{dels}, and such that $\supp \wh{h}\subset \rho_{n}B^{n}$,
where $\rho_{n}=q_{n/2}/(2\pi)$, we get $C_{A}=C_{L}$.

Recently, H.~Cohn and Y.~Zhao \cite{cozh14} proved that
\begin{equation}\label{2zv-5}
C_{A}\le C_{KL}.
\end{equation}
In 2016, $C_{A}$ was calculated when $n=8$ and $n=24$ (see \cite{mar1} and
\cite{mar2} respectively), and moreover, $\Delta_{n}=C_{A}$ for such $n$.

Our main result in this section is a new bound of Wiener's constant.
\begin{thm}\label{delll}
For $D\subset \T^{n}$ we have
\[
|D|^{-1}W_{n}(D)\le A_{\R^{n}}(D)\le (a_{\R^{n}}(D))^{-1}.
\]
\end{thm}

\begin{proof}
Let us first show that $A_{\R^{n}}(D)\le (a_{\R^{n}}(D))^{-1}$.
We have
\[
(A_{\R^{n}}(D))^{-1}=\sup \wh{h}(0),
\]
where $\wh{h}\ge 0$, $h|_{\R^{n}\setminus D}\le 0$, and $h(0)=1$. This problem
differs from the Tur\'an problem only by a less restrictive condition
$h|_{\R^{n}\setminus D}\le 0$ in place of $h|_{\R^{n}\setminus D}=0$. Hence,
$(A_{\R^{n}}(D))^{-1}\ge a_{\R^{n}}(D).$

It is enough to show that
\begin{equation}\label{W-A-bound}
|D|^{-1}W_{n}(D)\le A_{\R^{n}}(D).
\end{equation}
Let $h$ be an admissible function for the Delsarte problem. Consider
\[
g(x)=\sum_{\nu\in \Z^{n}}h(x+\nu),\quad x\in \T^{n}.
\]
This is a positive definite function on $\T^{n}$ and therefore $g(x)\le g(0)$
for any $x\in \T^{n}$. Since $D\subset \T^{n}$, we have $h(x+\nu)\le 0$ for
$x\in \T^{n}\setminus D$ and $h(\nu)\le 0$ for $\nu\ne 0$, $\nu\in \Z^{n}$.
Thus we have that $g|_{\T^{n}\setminus D}\le 0$ and $g(0)\le h(0)$.

Then, following the proof of Theorem~\ref{th-turan}, for any positive definite
function $f\in L^{1}(\T^{n})$ we obtain
\[
\int_{D}|f|^{2}\,dx\ge (g(0))^{-1}\int_{D}|f|^{2}g\,dx\ge
(g(0))^{-1}\int_{\T^{n}}|f|^{2}g\,dx.
\]
Hence, using \eqref{2zv}, for positive definitive functions $f$ and $g$, we
have
\[
\int_{\T^{n}}|f|^{2}g\,dx\ge \wh{g_0}\int_{\T^{n}}|f|^{2}\,dx
\]
and $\wh{g_0}=\wh{h}(0)=1$. Then
\[
\int_{D}|f|^{2}\,dx\ge
(g(0))^{-1}\int_{\T^{n}}|f|^{2}\,dx
\ge
(h(0))^{-1}\int_{\T^{n}}|f|^{2}\,dx,
\]
which implies $|D|^{-1}W_{n}(D)\le h(0)$.
Taking infimum over all $h$, we conclude the proof of \eqref{W-A-bound}.
\end{proof}

\section{Proofs of main results}\label{sec-proofs}

\begin{proof}[Proof of Proposition \ref{main0}]
First, since $F_+\subset L^1_+$, we always have $W_{n}^{+}(D) \le W_{n}(D)$.

To prove $W_{n}(D)\le W_{n}^{+}(D)$, let $f\in L_{+}^{1}(\T^{n})\cap L^{2}(D)$,
$f\ne 0$. We will show that to study the supremum in \eqref{W-F+} it is enough
to consider continuous functions $f$ satisfying $|f(0)|\le \|f\|_{L^{2}(D)}$.

Let us consider a non-negative continuous positive definite radial
function~$\varphi(x)$, $x\in \R^{n}$, such that $\supp \varphi\subset B^{n}$
and $\wh{\varphi}(0)=1$. For example, one can put
\[
\varphi(x)=h_{*}(x)/\wh{h}_{*}(0),
\]
where $h_{*}$ is given by \eqref{h*} with
$K=\tfrac{1}{2}B^{n}$.
Note that in this case, $h_{*}$ is radial since $\chi_{K}$ is radial.

For every small $\varepsilon>0$ such that $B_{\varepsilon}\subset D$ let us
define
\[
\varphi_{\varepsilon}(x)=\varepsilon^{-n}\varphi(\varepsilon^{-1}x).
\]
This function satisfies the following conditions
\begin{equation}\label{phi-cond}
\supp \varphi_{\varepsilon}\subset B_{\varepsilon},\qquad
\wh{\varphi_{\varepsilon}}(\xi)=\wh{\varphi}(\varepsilon \xi),\quad \xi\in
\R^{n},\qquad \wh{\varphi_{\varepsilon}}(0)=1.
\end{equation}
Now we set
\begin{equation}\label{psi-eps}
\psi_{\varepsilon}(x)=\sum_{\nu\in \Z^{n}}\varphi_{\varepsilon}(x+\nu), \quad x\in \T^{n}.
\end{equation}
Then, by \eqref{phi-cond}, this function satisfies the following conditions
\begin{itemize}
\item[(a)]\quad $\supp \psi_{\varepsilon}\subset B_{\varepsilon},$
\item[(b)]\quad $(\wh{\psi_{\varepsilon}})_\nu=\wh{\varphi_{\varepsilon}}(\nu)=\wh{\varphi}(\varepsilon
\nu),\quad \nu\in \Z^{n},$
\item[(c)]\quad $(\wh{\psi_{\varepsilon}})_0=1.$
\end{itemize}
Thus, $\psi_{\varepsilon}\in L_{+}^{1}(\T^{n})\cap C(\T^{n})$.

Define
\[
f_\varepsilon(x)=b(f*\psi_{\varepsilon})(x),\qquad x\in \T^{n},
\]
where
\[
\qquad
b^{-1}=\|\psi_{\varepsilon}\|_{L^{2}(\T^{n})}=\|\psi_{\varepsilon}\|_{L^{2}(B_{\varepsilon})}.
\]
Then we have that $(\wh{f}_{\varepsilon})_\nu=b\wh{f}_{\nu} (\wh{\psi_{\varepsilon}})_\nu
$.
Moreover, $f_{\varepsilon}\in F_{+}$, which follows from
 Young's inequality
\[
\|f_{\varepsilon}\|_{C(\T^{n})}\le
b\|f\|_{L^{2}(\T^{n})}\|\psi_{\varepsilon}\|_{L^{2}(\T^{n})}=\|f\|_{L^{2}(\T^{n})}.
\]
 Moreover,
\[
|f_{\varepsilon}(0)|=b\biggl|\int_{B_{\varepsilon}}f\psi_{\varepsilon}\,dx\biggr|\le
b\|f\|_{L^{2}(B_{\varepsilon})}\|\psi_{\varepsilon}\|_{L^{2}(B_{\varepsilon})}\le
\|f\|_{L^{2}(D)}.
\]

The function $\psi_{\varepsilon}$ is non-negative such that its mean value is
equal to $1$. Then by H\"{o}lder's inequality we get that
\[
\biggl|\int_{B_{\varepsilon}}f(x-y)\psi_{\varepsilon}(y)\,dy\biggr|^{2}\le
\int_{B_{\varepsilon}}|f(x-y)|^{2}\psi_{\varepsilon}(y)\,dy
\]
for any fixed $x$. Then it follows that
\[
\|f_{\varepsilon}\|_{L^{2}(D)}^{2}=
b^{2}\int_{D}\biggl|\int_{B_{\varepsilon}}f(x-y)\psi_{\varepsilon}(y)\,dy\biggr|^{2}\,dx\le
b^{2}\int_{B_{\varepsilon}}\psi_{\varepsilon}(y)g(y)\,dy,
\]
where
\[
g(y)=\int_{D}|f(x-y)|^{2}\,dx=(|f|^{2}*\chi_{D})(-y).
\]
The function $g$ is continuous in a neighborhood of the origin, since $|f|^{2}\in
L^{1}(\T^{n})$. Therefore, for any $\varepsilon'>0$ there exists
$\varepsilon>0$ such that
\[
|g(y)|\le (1+\varepsilon')g(0)=(1+\varepsilon')\|f\|_{L^{2}(D)}^{2},\quad
|y|<\varepsilon,
\]
which gives
\[
\|f_{\varepsilon}\|_{L^{2}(D)}^{2}\le b^{2}(1+\varepsilon')\|f\|_{L^{2}(D)}^{2}.
\]
By Parseval's identity,
\[
\|f_{\varepsilon}\|_{L^{2}(\T^{n})}^{2}=
b^{2}\sum_{\nu}(\wh{f}_{\nu})^{2}(\wh{\varphi}(\varepsilon \nu))^{2}.
\]
Moreover, by Hlawka's inequality \eqref{hlawka-ineq} (or by Theorem
\ref{th-turan}), we have
\[
\sum_{\nu}(\wh{f}_{\nu})^{2}=
\|f\|_{L^{2}(\T^{n})}^{2}\le \bigl|\tfrac{1}2D\bigr|^{-1} \|f\|_{L^{2}(D)}^{2}.
\]
Using $\wh{\varphi}^{2}(\varepsilon \nu)\le \wh{\varphi}^{2}(0) = 1$, we have
that $$\sum_{\nu}(\wh{f}_{\nu})^{2}(\wh{\varphi}(\varepsilon \nu))^{2}\to
\|f\|_{L^{2}(\T^{n})}^{2}$$ uniformly as $\varepsilon\to 0$. Hence for any
small $\varepsilon''>0$ we can find $\varepsilon>0$ such that
\[
\|f_{\varepsilon}\|_{L^{2}(\T^{n})}^{2}\ge b^{2}(1-\varepsilon'')\|f\|_{L^{2}(\T^{n})}^{2}.
\]
Finally,
\[
\frac{\|f\|_{L^{2}(\T^{n})}^{2}}{|D|^{-1}\|f\|_{L^{2}(D)}^{2}}\le
\frac{b^{-2}(1-\varepsilon'')^{-1}\|f_{\varepsilon}\|_{L^{2}(\T^{n})}^{2}}
{|D|^{-1}b^{-2}(1+\varepsilon')^{-1}\|f_{\varepsilon}\|_{L^{2}(D)}^{2}}\le
\frac{1+\varepsilon'}{1-\varepsilon''}\,W_{n}^{+}(D).
\]
Letting $\varepsilon',\varepsilon''\to 0$ gives the required result.
\end{proof}

\begin{proof}[Proof of Theorem \ref{main1}]
Combining property \eqref{A-hom}, Cohn--Zhao's estimate \eqref{2zv-5}, and
Theorem \ref{delll}, we arrive at
\begin{align*}
W_{n}(\delta B^{n})\le |\delta B^{n}|A_{\R^{n}}(\delta B^{n})&=
2^{n}|B^{n}|A_{\R^{n}}(2B^{n})\\
&=2^{n}C_{A}\le 2^{n}C_{KL},
\end{align*}
where $2^{n}C_{KL}=2^{(0.401\ldots+o(1))n}$.
\end{proof}

To prove Theorem \ref{main2}, we need the following

\begin{lem}\label{ll}
Let $\T^{n}=\T^{n_{1}}\times \T^{n_{2}}$, $D_{i}\subset \T^{n_{i}}$, $i=1,2$.
Then
\[
a_{\T^{n}}(D_{1}\times D_{2})=a_{\T^{n_{1}}}(D_{1})a_{\T^{n_{2}}}(D_{2}).
\]
\end{lem}

We note that for the spatial Tur\'an problem a similar fact is known, see
\cite{arbe02}. To make the paper self-contained we prove Lemma~\ref{ll} using
the same idea as in \cite{arbe02}.

\begin{proof}[Proof of Lemma \ref{ll}]
Let $D=D_{1}\times D_{2}$, $x=(x_{1},x_{2})\in \T^{n}$, where $x_{i}\subset
\T^{n_{i}}$, $i=1,2$.

We start by assuming $f_{i}(x_{i})$ to be admissible functions in the problem $a_{\T^{n_{i}}}(D_{i})$.
Then the function $f(x)=f_{1}(x_{1})f_{2}(x_{2})$ is also an admissible function in the problem
$a_{\T^{n}}(D)$, since $f(0)=1$, $f(x)=0$ for $x_{i}\notin D_{i}$, and
$\wh{f}_\nu=(\wh{f_{1}})_{\nu_{1}}(\wh{f_{2}})_{\nu_{2}}\ge 0$,
$\nu=(\nu_{1},\nu_{2})\in \Z^{n}$.

Therefore, we have $$a_{\T^{n}}(D)\ge \wh{f}_0=(\wh{f_{1}})_0(\wh{f_{2}})_0,$$
which gives that
\[
a_{\T^{n}}(D)\ge
a_{\T^{n_{1}}}(D_{1})a_{\T^{n_{2}}}(D_{2}).
\]
On the other hand, if $f(x)$ is an admissible function in the problem
$a_{\T^{n}}(D)$, then we define $f_{1}(x_{1})=f(x_{1},0)$ and
$$f_{2}(x_{2})=b^{-1}\int_{\T^{n_{1}}}f(x_{1},x_{2})\,dx_{1},\qquad
b=\int_{\T^{n_{1}}}f(x_{1},0)\,dx_{1}.$$ Let us show that the functions
$f_{i}(x_{i})$ are admissible in the problems $a_{\T^{n_{i}}}(D_{i})$.

First, we have that $\supp f_{i}\subset D_{i}$, $f_{1}(0)=f(0)=1$,
$f_{2}(0)=1$. The function~$f_{1}$ is positive definite since $f$ is positive
definite and
\[
f_{1}(x_{1})=f(x_{1},0)=\sum_{\nu_{1}\in
\Z^{n_{1}}}
(\wh{f_{1}})_{\nu_{1}}e(\nu_{1}x_{1}),
\]
where
\[
0\le (\wh{f_{1}})_{\nu_{1}}=\sum_{\nu_{2}\in
\Z^{n_{2}}}\wh{f}_{\nu_{1},\nu_{2}}\le \sum_{\nu\in \Z^{n}}\wh{f}_{\nu}=1.
\]
The function $f_{2}$ is positive definite since $b=(\wh{f_{1}})_0>0$ and
$(\wh{f_{2}})_{\nu_{2}}=b^{-1}\wh{f}_{0,\nu_{2}}\ge 0$.

We have
\[
\wh{f}_0=b(\wh{f_{2}})_0=(\wh{f_{1}})_0 (\wh{f_{2}})_0\le
a_{\T^{n_{1}}}(D_{1})a_{\T^{n_{2}}}(D_{2}).
\]
Thus
\[
a_{\T^{n}}(D)\le a_{\T^{n_{1}}}(D_{1})a_{\T^{n_{2}}}(D_{2}).
\]
\end{proof}

\begin{proof}[Proof of Theorem \ref{main2}]
To show the estimate of $W_{n}(D)$ from above we use Lemma \ref{ll} to get
\begin{equation}\label{an-a1}
a_{\T^{n}}([-\delta,\delta]^{n})=(a_{\T}([-\delta,\delta]))^{n},\quad
\delta\in (0,1/2).
\end{equation}

Let $\delta\in (0,1/2)$ and $\theta(\delta)=1-\delta
a_{\T}^{-1}([-\delta,\delta])$. Using Theorem \ref{th-turan}, and property
\eqref{an-a1}, we get
\begin{align*}
W_{n}(D)\le |D|(a_{\T}([-\delta,\delta]))^{-n}&=2^{n}(\delta
a_{\T}^{-1}([-\delta,\delta]))^{n}\\ &=2^{n}(1-\theta(\delta))^{n},
\end{align*}
completing the proof.
\end{proof}

\begin{proof}[Proof of Theorem \ref{main3}]
Let $D\subset \delta I^{n}$, $\delta\in (0,1/q)$, $q=2,3,\dots$,
$\Z_{q}=\Z/q\Z$, and $\Gamma=\{\nu/q\colon \nu \in \Z_{q}^{n}\}\subset \T^{n}$.
We have $|\Gamma|=q^{n}$ and
\[
s_{\nu}:=|\Gamma|^{-1}\sum_{\gamma\in \Gamma}e(\nu
\gamma)=\prod_{i=1}^{n}q^{-1}\sum\limits_{k=0}^{q-1}e\Bigl(\frac{\nu_i
k}{q}\Bigr)\in\{0,1\},
\]
where $\nu=(\nu_{1},\dots,\nu_{n})\in \Z^{n}$.

Let $0<\varepsilon<\min\{\delta,1/q-\delta\}$. Taking into account that
coordinates of lattice points $\Gamma$ are multiple of $1/q$, and
$\varepsilon<1/(2q)$, we have that
\[
(B_{\varepsilon}+\Gamma)\cap D=B_{\varepsilon}.
\]
Moreover, for any $\gamma, \gamma'\in \Gamma$ the sets $B_{\varepsilon}+\gamma$
and $B_{\varepsilon}+\gamma'$ are disjoint.

Now we will use the positive definite function $\psi_{\varepsilon}$ given by
\eqref{psi-eps} and satisfying $\supp \psi_{\varepsilon}\subset
B_{\varepsilon}$. We define
\begin{equation}\label{fG-def}
f(x)=|\Gamma|^{-1}\sum_{\gamma\in \Gamma}\psi_{\varepsilon}(x-\gamma).
\end{equation}
Then $f$ is a periodic function supported on $B_{\varepsilon}+\Gamma$ and such
that
\[
\wh{f}_{\nu}=(\wh{\psi_{\varepsilon}})_\nu s_{-\nu}\ge 0,\quad \nu\in \Z^{n}.
\]
This gives the positive definiteness of $f$.

Now we note that supports of the functions $\psi_{\varepsilon}(x-\gamma)$,
which are equal to $B_{\varepsilon}+\gamma$, are disjoint. Using this, we
obtain
\begin{equation}\label{f2-sum}
f^{2}(x)=|\Gamma|^{-2}\sum_{\gamma\in
\Gamma}\psi_{\varepsilon}^{2}(x-\gamma),\quad x\in \T^{n}.
\end{equation}
Integrating this and taking into account that
\[
\int_{\T^{n}}\psi_{\varepsilon}^{2}(x-\gamma)\,dx=
\int_{\T^{n}}\psi_{\varepsilon}^{2}(x)\,dx=
\|\psi_{\varepsilon}\|_{L^{2}(B_{\varepsilon})}^{2},
\]
we get
\begin{equation}\label{f2-int-T}
\int_{\T^{n}}f^{2}\,dx=|\Gamma|^{-2}\sum_{\gamma\in
\Gamma}\int_{\T^{n}}\psi_{\varepsilon}^{2}(x-\gamma)\,dx=
|\Gamma|^{-1}\|\psi_{\varepsilon}\|_{L^{2}(B_{\varepsilon})}^{2}.
\end{equation}
In light of $(B_{\varepsilon}+\Gamma)\cap D=B_{\varepsilon}$ we have
\begin{equation}\label{f2-int-D}
\int_{D}f^{2}\,dx=|\Gamma|^{-2}\int_{B_{\varepsilon}}\psi_{\varepsilon}^{2}\,dx=
|\Gamma|^{-2}\|\psi_{\varepsilon}\|_{L^{2}(B_{\varepsilon})}^{2}.
\end{equation}
Thus,
\begin{equation}\label{W-lower}
W_{n}(D)\ge
\frac{\int_{\T^{n}}f^{2}\,dx}{|D|^{-1}\int_{D}f^{2}\,dx}=|D||\Gamma|=|D|q^{n}.
\end{equation}
\end{proof}

\section{Wiener's inequality in $\R^{n}$ }

Let $f\colon \R^{n}\to \mathbb{C}$ be a positive definite function in the
following sense: $f$ is integrable and such that $\wh{f}\ge 0$. Then a formal
analogue of Wiener's inequality \eqref{wiener-f} given by
\[
\int_{\R^{n}}|f|^{2}\,dx\le C_{n}(D)\int_{D}|f|^{2}\,dx,
\]
for any origin-symmetric convex body $D\subset \R^{n}$, does not hold in
general. It is enough to consider the non-negative function \eqref{h*}
\[
f=|B_{r}|^{-1}\chi_{B_{r}}*\chi_{B_{r}}
\]
for sufficiently large $r$. Indeed, for $f(x)\le f(0)=1$, $\supp f\subset
B_{2r}$, we have
\[
\int_{\R^{n}}|f|^{2}\,dx=\int_{B_{2r}}|f|^{2}\,dx\ge
|B_{2r}|^{-1}\biggl(\int_{B_{2r}}f\,dx\biggr)^{2}.
\]
Moreover, similarly to \eqref{aR-chi},
\[
\int_{B_{2r}}f\,dx=\wh{f}(0)=|B_{r}|^{-1}\wh{\chi}_{B_{r}}^{2}(0)=
|B_{r}|.
\]
Then
\[
\int_{\R^{n}}|f|^{2}\,dx\ge 2^{-n}|B_{r}|\to \infty,\quad r\to \infty.
\]

\begin{thm}\label{thm-R}
Let $f\in L^{1}(\R^{n})$ be a positive definite function. Then for $D\subset \T^{n}$ we have
\begin{equation}\label{aR-Z}
\int_{\R^{n}}|f|^{2}\,dx\le (a_{\R^{n}}(D))^{-1}\int_{D+\Z^{n}}|f|^{2}\,dx.
\end{equation}
\end{thm}

Theorems \ref{th-turan} and \ref{thm-R} immediately imply
\begin{cor}[Hlawka's inequality in $\R^{n}$]\label{thm-R++}
Under conditions of Theorem~\ref{thm-R}, we have
\[
\int_{\R^{n}}|f|^{2}\,dx\le \bigl|\tfrac{1}{2}D\bigr|^{-1}\int_{D+\Z^{n}}|f|^{2}\,dx.
\]
\end{cor}
This result is new even in the one-dimensional case; a~weaker estimate was
proved in \cite[Th.~3.3]{kaonta94}.
\begin{rem}\label{thm-R+++}
Theorem~\ref{thm-R} and Corollary~\ref{thm-R++} hold for positive definite functions defined in the usual way:
 $f\in C(\R^{n})$ and
 \begin{equation}\label{f-mu} f(x)=f(0)\int_{\R^{n}}e(x\xi)\,d\mu(\xi),\quad f(0)\ge 0,\end{equation}
 where $\mu$
 is a finite positive measure on $R$, see \cite[Chap.~1]{ru62}.
\end{rem}

\begin{proof}[Proof of Theorem \ref{thm-R}]Let
$f\in L^{1}(\R^{n})$ and $\wh{f}\ge 0$. Suppose that $h\in L^{1}(\R^{n})$ is an
admissible function in the spatial Tur\'an problem, that is, $\wh{h}\ge 0$,
$\supp h\subset D$, and $h(0)=1$. Then $h(x)\le h(0)=1$ for $x\in \R^{n}$.

Let $g\in L^{1}(\T^{n})$ be a non-negative periodic positive definite function
satisfying $\wh{g}_{0}=1$ and $\supp g\subset B_{\varepsilon}$ for some small
positive $\varepsilon$. For example, we can take the
periodization of function
\eqref{h*}.
In this case, $$0\le \wh{g}_{\nu}\le \wh{g}_{0}=1,\qquad \nu\in \Z^{n}.$$
Now we set
\[
u(x)=\sum_{\nu\in \Z^{n}}\wh{g}_{\nu}h(x-\nu),\quad x\in \R^{n}.
\]
Then
\[
\wh{u}(\xi)=\wh{h}(\xi)\sum_{\nu\in \Z^{n}}\wh{g}_{\nu}e(\nu\xi)=
\wh{h}(\xi)g(\xi)\ge 0,\quad \xi\in \R^{n}.
\]
Let us estimate $I:=\int_{\R^{n}}|f|^{2}u\,dx$. First, we have that
\[
I=\sum_{\nu\in \Z^{n}}\wh{g}_{\nu}\int_{D+\nu}|f(x)|^{2}h(x-\nu)\,dx\le
\sum_{\nu\in \Z^{n}}\int_{D+\nu}|f|^{2}\,dx=\int_{D+\Z^{n}}|f|^{2}\,dx.
\]
On the other hand, we have that $fh\in L^{1}(\R^{n})\cap L^{2}(\R^{n})$,
$\wh{fh}=\wh{f}*\wh{u}\ge 0$, $\wh{f}\ge 0$, and
\begin{align*}
I&=\int_{\R^{n}}\ol{f(x)}\int_{\R^{n}}(\wh{fh})(\xi)e(\xi x)\,d\xi\,dx=
\int_{\R^{n}}(\wh{fh})(\xi)\ol{\wh{f}(\xi)}\,d\xi\\
&\ge \int_{B_{r}}(\wh{fh})(\xi)\wh{f}(\xi)\,d\xi,
\end{align*}
where $r>0$ is sufficiently large. Now we represent the latter integral as follows
\[
\int_{B_{r}}\wh{f}(\xi)\int_{\R^{n}}\wh{f}(\xi-\eta)\wh{u}(\eta)\,d\eta\,d\xi=
\int_{\R^{n}}\wh{u}(\eta)F_{r}(\eta)\,d\eta=
\int_{\R^{n}}\wh{h}(\eta)g(\eta)F_{r}(\eta)\,d\eta,
\]
where $$F_{r}(\eta)=\int_{B_{r}}\wh{f}(\xi)\wh{f}(\xi-\eta)\,d\xi.$$

Note that the functions $F_{r}$ and $\wh{h}$ are continuous at the
origin, and moreover,
\[
\int_{B_{\varepsilon}}g\,d\xi=\wh{g}_{0}=1
\]
and $g\ge 0$. Therefore, letting $\varepsilon\to 0$, we get by
 the second mean-value theorem that
\[
I\ge \wh{h}(0)F_{r}(0)=\wh{h}(0)\int_{B_{r}}(\wh{f}(\xi))^{2}\,d\xi=
\wh{h}(0)\|\wh{f}\|_{L^{2}(B_{r})}^{2}.
\]
Let us now let $r\to \infty$. We have
\[
I\ge
\wh{h}(0)\|\wh{f}\|_{L^{2}(\R^{n})}^{2}=\wh{h}(0)\|f\|_{L^{2}(\R^{n})}^{2}.
\]
Thus,
\[
\int_{\R^{n}}|f|^{2}\,dx\le (\wh{h}(0))^{-1}\int_{D+\Z^{n}}|f|^{2}\,dx.
\]
Maximizing $\wh{h}(0)$, we obtain \eqref{aR-Z}.\end{proof}

\begin{proof}[Proof of Remark \ref{thm-R+++}]

It is enough to consider the function
\[
f_{r}(x)=f(x)h_{*}(r^{-1}x),\quad x\in \R^{n},\quad r>0,
\]
where $h_{*}=|B_{1}|^{-1}\chi_{B_{1}}*\chi_{B_{1}}$, cf.~\eqref{h*}. Then
$f_{r}\in L^{1}(\R^{n})$ is a positive definite function with compact support.
Therefore, inequality \eqref{aR-Z} holds for such $f_{r}$. Since
$h_{*}(r^{-1}x)\to h_{*}(0)=1$ as $r\to \infty$ uniformly on any compact subset
of $\R^{n}$, we arrive at inequality \eqref{aR-Z} for $f$.
\end{proof}

\section{Final remarks}

1.~In $L^p(\T^{n})$, Wiener's theorem states that for $f\in L_{+}^{1}(\T^{n})$
one has
\[
\int_{\T^{n}}|f|^{p}\,dx\le C_{n,p}(D)\int_{D}|f|^{p}\,dx,
\]
if and only if $p$ is an even number; see \cite{wa69,sh75} for $n=1$. Here
$D\subset \T^{n}$ is an origin-symmetric convex body.

Similarly to Wiener's problem for $p=2$, we state the following question: to find
\[
W_{n,p}(D):=\sup_{f\in L_{+}^{1}(\T^{n})\setminus
\{0\}}\frac{\int_{\T^{n}}|f|^{p}\,dx }{{|D|^{-1}}\int_{D}|f|^{p}\,dx},
\]
where $p$ is an even integer.

First, we note that
\[
1\le W_{n,p}(D)\le W_{n, 2}(D)\le 2^{n},\qquad p=2k, \quad k\in \N.
\]
This follows immediately from the fact that $f^k$ is the positive definite
function. Thus, to estimate $W_{n,p}(D)$ from above, one can apply any upper
estimate of $W_{n, 2}(D)$, for example, given by Theorems \ref{main1} and
\ref{main2}.

Second, the proof of Theorem \ref{main3} can be easily modified to the case
$p>0$. This is because of the fact that crucial relations \eqref{f2-sum},
\eqref{f2-int-T}, and \eqref{f2-int-D} can be written as follows:
\begin{align*}
&f^{p}(x)=|\Gamma|^{-p}\sum_{\gamma\in
\Gamma}\psi_{\varepsilon}^{p}(x-\gamma),\\
&\int_{\T^{n}}f^{p}\,dx=|\Gamma|^{-p+1}\|\psi_{\varepsilon}\|_{L^{p}(B_{\varepsilon})}^{p},
\\
&
\int_{D}f^{p}\,dx=|\Gamma|^{-p}\|\psi_{\varepsilon}\|_{L^{p}(B_{\varepsilon})}^{p}.
\end{align*}
This immediately gives
\begin{thma}
Let $D\subset \delta I^{n}$ for $\delta\in (0,1/q)$, $q=2,3,\dots$. Then
\[
|D|q^{n}\le W_{n,p}(D).
\]
\end{thma}

2.~In \cite{astitu10}, it was proved that if $L^{p}(\T)\subset X$ and $X$ is a
solid, then
\[
L_{\mathrm{loc}+}^{p}(\T):=\biggl\{f\in L_{+}^{1}(\T)\colon
\int_{-\delta}^{\delta}|f|^p\,dx<\infty\biggr\}\subset X.
\]
Recall that a space of functions $X$ is called solid if it satisfies the
following property: For every $f=\sum c_{\nu}e(\nu x)$ in $X$, if another
function $g=\sum d_{\nu}e(\nu x)$ satisfies $|d_{\nu}|\le c_{\nu}$ for every
$\nu$ then $g$ is also in $X$. In particular, if $1<p\le 2$, then
\[
L_{\mathrm{loc}+}^{p}(\T^{n})\subset l^{p'}(\T^{n}).
\]
Therefore, one seeks the optimal constant
\[
\cW_{n,p}(D):= \sup_{f\in L_{\mathrm{loc}+}^{p}(\T^{n})\setminus
\{0\}}\frac{\bigl(\sum_{\nu}\wh{f}_{\nu}^{p'}\bigr)^{1/p'}}
{\bigl(|D|^{-1}\int_{D}|f|^{p}\,dx\bigr)^{1/p}},
\]
where $1<p\le 2$.

\smallbreak
3.~N.~Wiener in the 1930's \cite{wi34} was the first to study problem
\eqref{wiener-f} for trigonometric series with lacunary coefficients.
Later on,
this interesting problem received much attention by many authors including
A.~E.~Ingham \cite{in36},
\cite[Art.~45, Sec.~20, p.~224, (20.38)]{se91}, J.~D.~Vaaler
\cite{va85}, A.~Bonami and Sz.~R\'ev\'esz \cite{bore08} and, A.~Babenko and
V.~Yudin \cite{bayu15}.

\smallbreak
4.~In the case of $p=\infty$, Wiener's theorem becomes a well-known theorem of
Paley \cite{pa75}: if $f\in L_{\mathrm{loc}+}^{p}$ and $f$ is an even function,
then $f$ is continuous on $\T$ and its Fourier series converges uniformly and
absolutely. In this case Wiener's problem is closely related to a well-known
pointwise Tur\'an problem studied in \cite{arbebe03,kore06b,iviv10}.

The authors are grateful to V.I. Ivanov and the participants of the seminar on Approximation Theory at the Steklov Mathematical Institute
for useful comments.

\end{document}